\documentclass[a4paper,11pt]{article}

\usepackage{mdwlist}
\usepackage{amsfonts,amsmath,amsthm,amssymb}
\usepackage{epsfig}
\usepackage{enumerate}
\usepackage{a4wide}
\usepackage{makeidx}
\usepackage{color}
\usepackage{psfrag}

\newtheorem{theorem}{Theorem}[section]
\newtheorem{lemma}[theorem]{Lemma}

\newtheorem{definition}[theorem]{Definition}

\newtheorem{prop}[theorem]{Proposition}

\numberwithin{equation}{section}

\title{On the number of graphs not containing $K_{3,3}$ as a minor}
\date{}
\author{Stefanie Gerke\thanks{%Department of Mathematics
Royal Holloway, University of London, Egham, Surrey TW20 0EX UK,
{\tt stefanie.gerke@rhul.ac.uk}} \and Omer
Gim\'enez\thanks{Universitat Polit{\`e}cnica de Catalunya, Jordi
Girona 1-3,  08034 Barcelona,
{\tt\{omer.gimenez,marc.noy\}@upc.edu} } \and Marc
Noy\footnotemark[2]
  \and Andreas Wei{\ss}l\thanks{Google Switzerland GmbH,
Brandschenkestrasse 110, CH-8002 Zurich Switzerland, {\tt
weissl@google.com}}
  }

\newcommand{\deriv}[1]{\frac{\partial}{\partial {#1}}}

\begin{document}

%%% -------------------------------------------------
%%% --- title page ----------------------------------
%%% -------------------------------------------------
\maketitle

\numberwithin{equation}{section}

\begin{abstract}
  We derive precise asymptotic estimates for the number of
  labelled  graphs not containing $K_{3,3}$ as a minor, and also  for
 those which are edge maximal.  Additionally, we establish limit laws for parameters in random
  $K_{3,3}$-minor-free graphs, like the expected number of edges.
  To establish these results, we translate a decomposition for the corresponding
  graph class into equations for generating functions and use singularity
  analysis. We also find a precise estimate for the number of
  graphs not containing the graph $K_{3,3}$ plus an edge as a
  minor.
\end{abstract}


%%% -------------------------------------------------
%%% --- Introduction --------------------------------
%%% -------------------------------------------------
\section{Introduction}
\label{sec:Introduction}

We say that a graph is $K_{3,3}$-minor-free if it does not contain
the complete bipartite graph $K_{3,3}$ as a minor. In this paper
we are interested in the number of simple labelled
$K_{3,3}$-minor-free and maximal $K_{3,3}$-minor-free graphs,
where maximal means that adding any edge to such a graph yields a
$K_{3,3}$-minor. It is known that there exists a constant $c$,
such that there are at most $c^n n!$ $K_{3,3}$-minor-free graphs
on $n$ vertices. This follows from a result of Norine et
al.~\cite{nstw06}, which prove such a bound for all proper graph
classes closed under taking minors. This gives also an upper bound
on the number of maximal $K_{3,3}$-minor-free graphs as they are a
proper subclass of $K_{3,3}$-minor-free graphs.

In~\cite{msw05b}, McDiarmid, Steger and Welsh give conditions where an upper bound of the form
$c^n n!$ on the number of graphs $|{\cal C}_n|$ on $n$ vertices in graph class ${\cal C}$ yields
that $(|{\cal C}_n| / n!)^{\frac1n} \to c
> 0$ as $n \to \infty$. These conditions are satisfied for $K_{3,3}$-minor-free
graphs, but not in the case of maximal $K_{3,3}$-minor-free graphs as one
condition requires that two disjoint copies of a graph of the class in question
form again a graph of the class.

Thus we know that there exists a growth constant $c$ for
$K_{3,3}$-minor-free graphs, but not its exact value. For maximal
$K_{3,3}$-minor-free graphs we only have an upper bound. Lower
bounds on the number of graphs in our classes can be obtained by
considering (maximal) planar graphs. Due to Kuratowski's
theorem~\cite{k30} planar graphs are $K_{3,3}$- \emph{and}
$K_5$-minor-free. Hence, the class of (maximal) planar graphs is
contained in the class of maximal $K_{3,3}$-minor-free graphs and
we can use the number of planar graphs and the number of
triangulations as lower bounds. Determining the number (of graphs
of sub-classes) of planar graphs has attracted considerable
attention \cite{bgw02,gn05,bgkn05,blmk06} in recent years.
Gim\'enez and Noy \cite{gn05} obtained precise asymptotic
estimates for the number of planar graphs. Already in 1962, the
asymptotic number of triangulations was given by Tutte~\cite{t62}.
Investigating how much the number of planar graphs
(triangulations) differs from (maximal) $K_{3,3}$-minor-free
graphs was also a main motivation for our research.

%Here we will derive asymptotic estimates for the number of (maximal) $K_{3,3}$-minor-free graphs.
%For $K_{3,3}$-minor-free graphs we consider $3$-connected, $2$-connected, connected and not
%necessarily connected graphs.

%For $3$-connected $K_{3,3}$-minor-free graphs, the change from planar graphs can
%be easily described: it follows from a theorem of Wagner~\cite{w37} that the set
%of $3$-connected $K_{3,3}$-minor-free graphs consists of all $3$-connected planar graphs and the
%complete graph on $5$ nodes. Thus, on this connectivity level our graph class differs only in the
%existence of one additional graph from planar graphs. But as we will show, adding $K_5$ to the set
%of $3$-connected graphs yields a slightly larger exponential growth rate for $2$-connected,
%connected, and not necessarily connected $K_{3,3}$-minor-free graphs. It also slightly changes
%other parameters, for instance the expected number of edges in a random $K_{3,3}$-minor-free
%graph. For maximal $K_{3,3}$-minor-free graphs the growth rate also increases slightly compared to
%triangulations.

In this paper
 we derive precise asymptotic estimates for the number of simple labelled
$K_{3,3}$-minor-free and maximal $K_{3,3}$-minor-free graphs on
$n$ vertices, and we establish several limit laws for parameters
in random $K_{3,3}$-minor-free graphs. More precisely, we show
that the number $g_n$, $c_n$, and $b_n$ of not necessarily
connected, connected and $2$-connected $K_{3,3}$-minor-free graphs
on $n$ vertices, and the number $m_n$ of maximal
$K_{3,3}$-minor-free graphs on $n$ vertices satisfy
\begin{eqnarray*}
  g_n & \sim & \alpha_g \; n^{-7/2} \; \rho_g^{-n} \; n!, \\
  c_n & \sim & \alpha_c \; n^{-7/2} \; \rho_c^{-n} \; n!, \\
  b_n & \sim & \alpha_b \; n^{-7/2} \; \rho_b^{-n} \; n!, \\
  m_n & \sim & \alpha_m \; n^{-7/2} \; \rho_m^{-n} \; n!
\end{eqnarray*}
where $\alpha_g \doteq 0.42643 \cdot 10^{-5}$, $\alpha_c \doteq
0.41076 \cdot 10^{-5}$, $\alpha_b \doteq 0.37074 \cdot 10^{-5}$,
$\alpha_m \doteq 0.25354 \cdot 10^{-3}$, and $\rho_c^{-1} =
\rho_g^{-1} \doteq 27.22935$, $\rho_b^{-1} \doteq 26.18659$, and
$\rho_m^{-1} \doteq 9.49629$ are analytically computable
constants. Moreover, we derive limit laws for $K_{3,3}$-minor-free
graphs, for instance we show that the number of edges is
asymptotically normally distributed with mean $\kappa n$ and
variance $\lambda n$, where $\kappa \doteq 2.21338$ and $\lambda
\doteq 0.43044$ are analytically computable constants. Thus the
expected number of edges is only slightly larger than for planar
graphs~\cite{gn05}.

To establish these results for $K_{3,3}$-minor-free graphs, we follow the
approach taken for planar graphs~\cite{bgw02, gn05}: we use a well-known
decomposition along the connectivity structure of a graph, i.e.\ into connected,
$2$-connected and $3$-connected components, and translate this decomposition
into relations of our generating functions. This is possible as the
decomposition for $K_{3,3}$-minor-free graphs which is due to Wagner~\cite{w37}
fits well into this framework. Then we use singularity analysis to obtain
asymptotic estimates and limit laws for several parameters from these equations.

For maximal $K_{3,3}$-minor-free graphs the situation is quite
different, as the decomposition which is again due to Wagner has
further constraints (it restricts which edges can be used to merge
two graphs into a new one). The functional equations for the
generating functions of \emph{edge-rooted} maximal graphs are easy
to obtain but in order to go to unrooted graphs, special
integration techniques based on rational parametrization of
rational curves are needed. This is the most innovative part of
the paper with respect to previous work, specially with respect to
the techniques developed in \cite{gn05}. As a result,  we can
derive equations for the generating functions which involve the
generating function for triangulations derived by
Tutte~\cite{t62}, and deduce precise asymptotic estiamates.

In the subsequent sections, we proceed as follows.
%First, we introduce some necessary notation in Section~\ref{sec:Notation}.
First, we turn to maximal $K_{3,3}$-minor-free and
$K_{3,3}$-minor-free graphs in
Sections~\ref{sec:MaximalK33MinorFree} and~\ref{sec:K33MinorFree}
respectively. In each of these sections, we will first derive
relations for the generating functions based on a decomposition of
the considered graph class and then apply singularity analysis to
obtain asymptotic estimates for the number (and properties) of the
graphs in these classes. The last section contains the enumeration
of graphs not containing $K_{3,3}^+$ as a minor, where $K_{3,3}^+$
is the graph obtained from $K_{3,3}$ by adding an edge.

Throughout the paper variable $x$ marks vertices and variable $y$
marks edges. Unless we specify the contrary, the graphs we
consider are labelled and the corresponding generating functions
are exponential. We often need to distinguish an atom of our
combinatorial objects; for instance we want to mark a vertex in a
graph as a root vertex. For the associated generating function
this means taking the derivative with respect to the corresponding
variable and multiplying the result by this variable. To simplify
the formulas, we use the following notation. Let $G(x, y)$ and
$G(x)$ be generating functions, then we abbreviate $G^{\bullet}(x,
y) = x \deriv{x} G(x, y)$ and $G^{\bullet}(x) = x \deriv{x} G(x)$.
Additionally, we use the following standard  notation: for a graph
$G$  we denote by $V(G)$ and $E(G)$ the vertex- and edge-set of
$G$.

%%% -------------------------------------------------
%%% --- Maximal K_{3,3}-minor-free Graphs -----------
%%% -------------------------------------------------
\section{Maximal $K_{3,3}$-minor-free graphs}
\label{sec:MaximalK33MinorFree}

%\input{maximal.ltx}

% \section{Maximal $K_{3,3}$-minor-free graphs}

Already in the 1930s, Wagner \cite{w37} described precisely the
structure of maximal $K_{3,3}$-minor-free graphs. Roughly speaking
his theorem states that a maximal graph not containing $K_{3,3}$
as a minor is formed by gluing planar triangulations and the
exceptional graph $K_5$ along edges, in such a way that no two
different triangulations are glued along an edge. Before we state
the theorem more precisely, we introduce the following notation
(similar to \cite{t99}, see also Section~\ref{sec:Decomposition}).

\begin{definition}
  Let $G_1$ and $G_2$ be graphs with disjoint vertex-sets, where each edge is
  either colored blue or red. Let $e_1 = (a, b) \in E(G_1)$ and $e_2 = (c, d)
  \in E(G_2)$ be an edge of $G_1$ and $G_2$ respectively. For $i=1,2$ let $G_i'$
  be obtained by deleting edge $e_1$ and coloring edge $e_2$ blue if $e_1$ and
  $e_2$ were both colored blue and red otherwise. Let $G$ be the graph obtained
  from the union of $G_1'$ and $G_2'$ by identifying vertices $a$ and $b$ by $c$
  and $d$ respectively. Then we say that $G$ is a \emph{strict
  $2$-sum} of $G_1$ and $G_2$. We say that a strict $2$-sum is \emph{proper} if
  at least one of the edges $e_1$ and $e_2$ is blue.
\end{definition}

\begin{theorem}[Wagner's theorem \cite{w37}]
  \label{thm:wagner}
  Let $\mathcal T$ denote the set of all labeled planar triangulations where
  each edge is colored red. Let each edge of the complete graph $K_5$ be
  colored blue. A graph is maximal $K_{3,3}$-minor-free if
  and only if it can be obtained from planar triangulations and $K_5$ by proper,
  strict $2$-sums.
\end{theorem}

\smallskip
Let  $\mathcal{A}$ be  the family of maximal graphs not containing
$K_{3,3}$ as a minor.  Let $\mathcal H$ be the family of
edge-rooted graphs in $\mathcal A$, where the root belongs to a
triangulation, and let $\mathcal F$ be edge-rooted graphs in
$\mathcal A$, where the root does not belong to a triangulation.

 Let $T_0(x,y)$ be the generating function (GF for short)
of edge-rooted planar triangulations (the root-edge is included),
and let $K_0(x,y)$ be the GF of edge-rooted $K_5$ (the root-edge
is not included). Let $A(x,y), F(x,y), H(x,y)$ be the GFs
associated respectively to the families
$\mathcal{A},\mathcal{F},\mathcal{H}$. In all cases the two
endpoints of the root edge do not bear labels, and the root edge
is oriented; this amounts to multiplying the corresponding GF by
$2/x^2$. Notice that
 $$
    K_0 = {2 \over x^2} {\partial \over \partial y} \left(y^{10}{x^5\over5!} \right)
      = y^9 {x^3\over6}.
$$
Since edge-rooted graphs from $\mathcal{A}$ are the disjoint union
of $\mathcal{H}$ and $\mathcal{F}$, we have
\begin{equation}\label{root}
 H(x,y) + F(x,y)= {2 \over x^2} y\,{\partial A(x,y) \over \partial y}.
\end{equation}
The fundamental equations that we need are the following:
\begin{eqnarray}
  H &=& T_0(x,F) \label{H} \\
  F &=& y\exp\left(K_0(x,H+F) \right) \label{F}
\end{eqnarray}
The first  equation means that a graph in $\mathcal{H}$ is
obtained by substituting every edge in a planar triangulation by
an edge-rooted graph whose root does not belong to a triangulation
(because of the statement of Wagner's theorem). The second
equation means  that a graph in $\mathcal{F}$  is obtained by
taking (an unordered) set of $K_5$'s in which each edge is
substituted by an edge-rooted graph either in $\mathcal{H}$ or in
$\mathcal{F}$.

Eliminating $H$ we get the equation
\begin{equation}\label{eqF}
    F = y \exp\left(K_0(x,F+ T_0(x,F))\right).
\end{equation}
Hence, for fixed $x$,
 \begin{equation}\label{inverse}
 \psi(u) = u \exp\left(-K_0(x,u+T_0(x,u)\right) =
 u \exp\left( -{x^3 \over 6}(u+T_0(x,u))^9 \right)
\end{equation}
is the functional inverse of $F(x,y)$.

In order to analyze $F$ using Equation~(\ref{F}) we need to know
the series $T_0(x,y)$ in detail. Let $T_n$ be the number of
(labeled) planar triangulations with $n$ vertices. Since a
triangulation has exactly $3n-6$ edges, we see that
 $$
 T(x,y) = \sum T_n y^{3n-6} {x^n \over n!}
 $$
is the GF of triangulations. And since
 $$
 T_0(x,y) = {2\over x^2}  y {\partial T(x,y) \over \partial y},
 $$
 it is enough to study $T$.

Let now $t_n$ be the number of rooted (unlabeled) triangulations
with $n$ vertices in the sense of Tutte  and let $t(x) = \sum t_n
x^n $ be the corresponding \emph{ordinary} GF.
 We know (see \cite{t62}) that $t(x)$ is equal to
 $$
 t = x^2 \theta (1-2\theta)
$$
where $\theta(x)$ is the algebraic function defined by
 $$
    \theta (1-\theta)^3 = x.
$$
It is known that the dominant singularity of $\theta$ is at
$R=27/256$ and $\theta(R) = 1/4$.

There is a direct relation between the numbers $T_n$ and $t_n$. An
unlabeled rooted triangulation can be labeled in $n!$ ways, and a
labeled triangulation ($n\ge4$) can be rooted in $4(3n-6)$ ways,
since we have $3n-6$ possibilities for choosing the root edge, two
for orienting the edge, and two for choosing the root face. Hence
we have
$$t_n n! = 4(3n-6) T_n, \quad n\ge4, \qquad  t_3 = T_3=1.$$
The former identity implies easily the following equation
connecting the exponential GF $T(x,y)$ and the ordinary GF $t(x)$:
 $$
 y {\partial T \over \partial y} = y^3{x^3\over 4} + {t(xy^3) \over
 4y^6}.
  $$
Hence we have
% \begin{equation}\label{T0}
$$
 T_0(x,y) = {2 \over x^2}  y {\partial T \over \partial y} =  y^3{x\over 2} + {t(xy^3) \over
 2x^2y^6}.
 $$
%\end{equation}
  %
The last equation is crucial since it expresses $T_0$ in terms of
a known algebraic function. It is convenient to rewrite it as
\begin{equation}\label{t0}
 T_0(x,y) = y^3{x\over 2} + {1 \over 2}L(x,y)(1-2L(x,y)),  \qquad \hbox{where } L(x,y) = \theta(xy^3).
\end{equation}
The series $L(x,y)$ is defined through the algebraic equation
\begin{equation}\label{L}
    L(1-L)^3 - xy^3 = 0.
\end{equation}
This defines a rational curve, i.e.\ a curve of genus zero, in the
variables $L$ and $y$ (here $x$ is taken as a parameter)  and
admits the rational (actually polynomial) parametrization
\begin{equation}\label{parametric}
    L = \lambda(t) = -{t^3 \over x^2}, \qquad y = \xi(t) = - {t^4 + x^2t \over x^3}.
\end{equation}
This is a key fact that we use later.

\bigskip

We have set up the preliminaries needed in order to analyze the
series $A(x,y)$, which is the main goal of this section. From
(\ref{root}) it follows that
 $$
 A(x,y) = {x^2\over2} \int_0^y {H(x,t)\over t } \,dt +
 {x^2\over2} \int_0^y {F(x,t)\over t } \,dt.
 $$
The following lemma expresses $A(x,y)$ directly in terms of $H$
and $F$ \emph{without} integrals.

\begin{lemma}\label {int}
The generating function $A(x,y)$ of maximal graphs not containing
$K_{3,3}$ as a minor can be expressed as

\begin{equation}\label{explicit}
A(x,y)=\frac{-x^2}{60}\left(27(H+F)\log\left(\frac{F}{y}\right)+10L+20L^2
+15\log(1-L)-30F-5xF^3\right)
\end{equation}
where $L=L(x,F(x,y))$, $H=H(x,y)$ and $F=F(x,y)$ are defined
through (\ref{L}), (\ref{H}) and (\ref{F}).

%% \begin{equation}\label{explicit}
%%  A(x,y) ={x^2\over2}\left( \log(y)(H+F) + {x^3 \over 6} {(H+F)^{10} \over
%%  10} - J(x,\zeta(F)) \right),
%% \end{equation}
%% where $H$ and $F$ are defined through  (\ref{H}) and (\ref{F}),
%%   $$
%% J(x,t) =
%%  \left( -5/2\,{\frac {{t}^{6}}{{x}^{4}}}-1/2\,{\frac {{t}^{12}}{{x}^{8
%% }}}-{\frac {{t}^{3}}{{x}^{2}}}-{\frac {{t}^{4}}{{x}^{3}}}-{\frac
%% {t}{x }}-3/2\,{\frac {{t}^{9}}{{x}^{6}}} \right) \ln  \left(
%% {\frac {t
%%  \left( -{t}^{3}-{x}^{2} \right) }{{x}^{3}}} \right) $$
%%  $$ +7/6\,{\frac {{t}
%% ^{6}}{{x}^{4}}}-1/2\log(x^2+t^3)-1/6\,{\frac
%% {{t}^{3}}{{x}^{2}}}+{\frac {t}{x}}+{\frac
%% {{t}^{4}}{{x}^{3}}}+1/2\,{ \frac {{t}^{9}}{{x}^{6}}}+1/6\,{\frac
%% {{t}^{12}}{{x}^{8}}},
%% $$
%% and $\zeta(s)$ is an algebraic function which, for fixed $x$, is
%% defined by
%% \begin{equation}\label{zeta}
%%   \zeta^4 + x^3 \zeta^2 - x^3 s = 0.
%% \end{equation}
\end{lemma}

\begin{proof}
Integration by parts gives
 \begin{equation}\label{eq:beginninglemma}
 A(x,y)=\frac{x^2}{2} \int_0^y {H(x,t) +F(x,t) \over t} \, dt =
 \frac{x^2}{2}(H + F) \log(y) - \frac{x^2}{2}I
 \end{equation}
where
$$
  I =  \int_0^y \log(t)\left(H'(x,t) + F'(x,t) \right)\,dt
 $$
and derivatives are with respect to the second variable.
 Because of (\ref{inverse}), the change of variable $s=F(x,t)$ gives $t=\psi(s)$ and
$$
\log(t) = \log(s) - {x^3 \over 6} \left(s+T_0(x,s)^9\right).
$$
Since $H=T_0(x,F)$ we have $H'=T_0'(x,F) F'$ and so
\begin{eqnarray*}
   I &=& \int_0^{F} \left(\log(s) -{x^3\over
  6}(s+T_0(x,s))^9\right) \left(1+T_0'(x,s)\right) \, ds \\
  && = -{x^3 \over 6} {(F+T_0(x,F))^{10} \over 10}
  + \int_0^{F} \log(s) \left(1+T_0'(x,s)\right) \, ds \\
  && = -\frac{1}{10}(H+F)\log\left(\frac{F}{y}\right)
  + \int_0^{F} \log(s) \left(1+T_0'(x,s)\right) \, ds \\
\end{eqnarray*}
where the last equality follows from Equation (\ref{F}).

It remains to compute the last integral. From (\ref{t0}) it
follows easily that
\begin{equation}\label{T0p}
T_0' = {3y^2x \over 2} \left(1+ {1\over (1-L)^2}\right).
\end{equation}
Now we change variables according to (\ref{parametric}) taking $s
= \xi(t)$, so that $L=\lambda(t)$. Let $\zeta$ be the inverse
function of $\xi$, so that $t = \zeta(s)$. Observe that $\zeta(s)$
satisfies
$$
  \zeta^4 + x^2 \zeta + x^3 s = 0.
$$
Then we have
\begin{eqnarray*}
   &&\int_0^{F} \log(s)\left(1+T_0'(x,s)\right) \, ds \\
    &=& \int_0^{\zeta(F)} \log(\xi(t))\left(1+ {3\xi(t)^2 x\over2}
    \left(1 + {1 \over (1-\lambda(t))^2}\right)\right) \xi'(t)\,dt
\end{eqnarray*}
After substituting the expressions for $\xi(t)$ and $\lambda(t)$,
the integrand in the last integral is equal to
$$
f(x,t) = -{1 \over 2x^8}\, \left( 4\,{t}^{3}+{x}^{2} \right)
\left( 2\,{x}^{5}+3\,{t}^{8} +6\,{t}^{5}{x}^{2}+6\,{t}^{2}{x}^{4}
\right) \ln  \left( -{\frac {
  {t}^{4}+{x}^{2}t  }{{x}^{3}}} \right).
$$

The function $f(x,t)$ can be integrated in elementary terms,
resulting in
\begin{align*}
\int_0^{\zeta(F)} f(x,t)dt =&
\left(-\frac{5\zeta^6}{2x^4}-\frac{\zeta^{12}}{2x^8}
-\frac{\zeta^3}{x^2}-\frac{\zeta^4}{x^3}-\frac{\zeta}{x}
-\frac{3\zeta^9}{2x^6}\right)\log\left(-\frac{\zeta^4+x^2\zeta}{x^3}\right) \\
&
+\frac{7\zeta^6}{6x^4}-\frac{\zeta^3}{6x^2}+\frac{\zeta}{x}+\frac{\zeta^4}{x^3}
+\frac{\zeta^9}{2x^6}+\frac{\zeta^{12}}{6x^8}-\frac{1}{2}\log\left(1+\frac{\zeta^3}{x^2}\right),
\end{align*}
where $\zeta=\zeta(F)$. All terms with $\zeta$ are powers of
either of the two forms
$$
-\frac{\zeta^4+x^2\zeta}{x^3}=\xi(\zeta(F))=F, \qquad
-\frac{\zeta^3}{x^2}=\lambda(\zeta(F))=L(x,F),
$$
so we can write the integral of $f(x,t)$ explicitly in terms of
$F$ and $L=L(x,F)$,
$$
\left(-\frac{1}{2}L^4+\frac{3}{2}L^3-\frac{5}{2}L^2+L+F\right)
\log(F)+\frac{L^4}{6}-\frac{L^3}{2}+\frac{7L^2}{6}
+\frac{L}{6}+\frac{\log(1-L)}{2}-F.
$$

We simplify this expression further using that, according to
Equations (\ref{H}), (\ref{t0}) and (\ref{L}),
\begin{equation}\label{H_LF}
H=T_0(x,F)=\frac{1}{2}\left(xF^3+L(1-2L)\right)=\frac{1}{2}(-L^4+3L^3-5L^2+2L).
\end{equation}

Obtaining the final expression for $A(x,y)$ is just a matter of
going back to Equation~(\ref{eq:beginninglemma}) and adding up all
terms.
\end{proof}

%The constant term must be corrected a little so that $J(x,0)=0$,
%but this is easily solved by taking instead
% $$
% J -1/2\log(x^2+t^3)+1/2\log(1+t^3/x^2).
% $$

Summarizing, we have an explicit expression for $A$ in terms of
$x$, $y$, $H(x,y)$ and $F(x,y)$ which involves only elementary
functions and the algebraic function $L(x,y)$. Moreover, note that
$H(x,y)$ can be written in terms of $L(x,F)$ by
Equation~(\ref{H_LF}). Our goal is to carry out a full singularity
analysis of the univariate GF $A(x) = A(x,1)$. To this end we
first perform singularity analysis on $F(x) = F(x,1).$

\begin{lemma}\label{lemma-F}
The dominant singularity of $F(x)$ is the unique $\rho>0$ such
that $\rho F(\rho)^3 = 27/256$. The approximate value is
$\rho\approx 0.10530385$. The value $F(\rho)\approx 1.0005216$ is
the solution of
\begin{equation}\label{eq:F0}
  t =
  \exp\left(\frac{27^3}{6 \cdot 256^3}\left(1+\frac{59}{512t}\right)^9\right).
\end{equation}
\end{lemma}

\begin{proof}

The function $F(x)$ satisfies
\begin{equation}\label{phi}
 \Phi(x,F) = \exp\left({x^3 \over 6} \left(F+T_0(x,F)\right)^9
 \right)-F.
\end{equation}
Thus the dominant singularity $\rho$ of $F(x)$ may come from $T_0$
or from a branch point when solving~(\ref{phi}). Assume that there
is no such branch point. Then, since $L(x,y)=\theta(xy^3)$ and the
dominant singularity of $\theta$ is at $27/256$, we have that
$L(\rho,F(\rho))=1/4$ and $\rho F(\rho)^3=27/256$. Substituting in
$\Phi(x,F)=0$ we obtain Equation~(\ref{eq:F0}), where $t$ stands
for $F(\rho)$. The approximate value is $t\approx 1.0005216$,
which gives $\rho\approx 0.10530385$, slightly smaller than
$R=27/256=0.10546875$.

We  now prove that there is no branch point when solving $\Phi$.
If this were the case, then there would exist $\tilde{\rho}\leq
\rho$ such that $\partial_F \Phi(\tilde{\rho},F(\tilde{\rho}))=0$,
where
\begin{equation}\label{partialF-Phi}
{\partial \over \partial F}\Phi(x,F(x))
=\frac{3}{1024}(-3L^2+3L+2F+3xF^3)x^3 (2F+xF^3+L-2L^2)^8-1.
\end{equation}
follows by differentiating Equation~(\ref{phi}), by using
$\Phi(x,F(x))=0$ and Equations~(\ref{L}), (\ref{T0p}), and
(\ref{H_LF}).

Consider $\partial_F\Phi(x,F,L)$ as a function of three
independent variables, where $x\geq 0$, $F\geq 1$ and $0\leq L\leq
1/4$. It follows easily that it is an increasing function on any
of them. Hence
$$
        0=\partial_F\Phi(\tilde{\rho}, F(\tilde{\rho}), L(\tilde{\rho}, F(\tilde{\rho})))\leq
          \partial_F\Phi(\rho, F(\tilde{\rho}), 1/4),
$$
since, by assumption, $\tilde{\rho}\leq\rho$. On the other hand
$\partial_F\Phi(\rho, t, 1/4)\approx -0.9939$, so by comparing
this with $\partial_F\Phi(\rho, F(\tilde{\rho}), 1/4)$ we deduce
that $t<F(\tilde{\rho})$. Since $1=F(0)<t$, by continuity of
$F(x)$ there exists  a value $x\in(0,\tilde{\rho})$ such that
$F(x)=t$, and by the monotonicity of $\Phi(x,F)$ for fixed $F$
there is a unique solution $x$ to $\Phi(x,t)=0$. This solution is
precisely $x=\rho$, contradicting $\tilde{\rho}\leq \rho$.
\end{proof}

\begin{prop}\label{prop-F}
Let $\rho$ and $t$ be as in Lemma~\ref{lemma-F}. The singular
expansions of $F(x)$ at $\rho$ is

$$
 F(x) = t + F_2 X^2 + F_3 X^3 + {\mathcal O}(X^4),
$$
where $X=\sqrt{1-x/\rho}$, and $F_2$ and $F_3$ are given by
$$
F_2= \frac{12t(128t+71)\log\left(t\right)}{Q}, \qquad F_3=
\frac{96\sqrt{6}\ t\log(t)M^{3/2}}{Q^{5/2}}
$$
$$
M= 531\log(t)+512t+59, \qquad Q= 9(225+512t)\log(t)-512t-59.
$$
\end{prop}

\begin{proof}

To obtain the coefficients of the singularity expansion, including
the fact that $F_1=0$, we apply indeterminate coefficients $F_i$,
$L_i$ of $X^i$ to Equations (\ref{phi}) and
$$
    L(x)(1-L(x))^3 - xF(x)^3=0,
$$
where $X=\sqrt{1-x/\rho}$, so that $x=\rho(1-X^2)$. These
calculations are tedious, but can be done with a computer algebra
system such as {\sc Maple}.
\end{proof}

\begin{prop}\label{th-A}
Let $\rho$ and $t$ be as in Lemma~\ref{lemma-F}. The dominant
singularity of $A(x)$ is $\rho$, and its singular expansion at
$\rho$ is

 $$
 A(x) = A_0 + A_2 X^2 + A_4 X^4 + A_5 X^5 + {\mathcal O}(X^6),
 $$
where $X = \sqrt{1-x/\rho}$ and $A_0$, $A_2$, $A_4$ and $A_5$ are
given by

\begin{align*}
A_0 = & -\frac{3 C}{20 t^6}\left(4608\log(t)t+531\log(t)+2560\log(3/4)-5120t+550\right) \\
A_2 = & \frac{C}{4 t^6} \left(4608\log(t)t+531\log(t)+3072\log(3/4)-6144t+542\right) \\
A_4 = & \frac{3 C}{t^6} \left(16Q^{-1}\log(t)(128t+71)^2+59\log(t)+2^9(\log(t)t-2t+\log(3/4))+26\right) \\
A_5 = & \frac{40 \sqrt{6} C}{3 t^6}\left(\frac{M}{Q}\right)^{5/2}
\end{align*}
where $C=3^5/2^{25}$, and $M$ and $Q$ are as in
Proposition~\ref{prop-F}.
\end{prop}

\begin{proof}

We just compute the singular expansion
 $$
 A(x) = \sum_{k\ge 0} A_k X^k,
 $$
by plugging the expansions for $F(x)$ and $L(x)$ of
Proposition~\ref{lemma-F} in (\ref{explicit}).
\end{proof}

\begin{theorem}\label{A}
The number $A_n$ of maximal graphs with $n$ vertices not
containing $K_{3,3}$ as a minor is asymptotically
 $$
 A_n \sim a\cdot n^{-7/2} \gamma^n n!,
 $$
where $\gamma = 1/\rho \approx 9.49629$ and $a=-15A_5/8\pi\simeq
0.25354\cdot 10^{-3}$.
\end{theorem}

\begin{proof}
This is a standard application of singularity analysis (see for
example Corollary~VI.1 of \cite{fs05})  to the singular expansion
of $A(x)$ obtained in the previous lemma.
\end{proof}

%%% -------------------------------------------------
%%% --- K_{3,3}-minor-free Graphs -------------------
%%% -------------------------------------------------
\section{$K_{3,3}$-minor-free graphs}
\label{sec:K33MinorFree}

In this section, we derive the asymptotic number of $K_{3,3}$-minor-free graphs
and properties of random $K_{3,3}$-minor-free graphs.

%%% -------------------------------------------------
%%% --- Notation ------------------------------------
%%% -------------------------------------------------
%\subsection{Notation} \label{sec:Notation}

%%% -------------------------------------------------
%%% --- Decomposition and GFs -----------------------
%%% -------------------------------------------------
\subsection{Decomposition and Generating Functions}
\label{sec:Decomposition}

Let $G(x, y)$, $C(x, y)$ and $B(x, y)$ denote the exponential
generating functions of simple labelled \emph{not necessarily
connected}, \emph{connected} and \emph{$2$-connected}
$K_{3,3}$-minor-free graphs respectively.
%, where $x$ marks vertices and $y$ marks edges.
We will use the additional variable $q$ to mark the number of $K_5$'s used in
the ``construction process'' of a $K_{3,3}$-minor-free graph (see below for a
more precise explanation), but we won't give it explicitly in the argument list
of our generating functions to simplify expressions.

We want to apply singularity analysis to derive asymptotic estimates for the
number of $K_{3,3}$-minor-free graphs. To achieve this, we first present a
decomposition of this graph class which is due to Wagner~\cite{w37}. Then we
will translate it into relations of our generating functions.

As seen in Theorem \ref{thm:wagner} above, Wagner \cite{w37} characterized the
class of maximal $K_{3,3}$-minor-free graphs. As a direct consequence we also
obtain a decomposition for $K_{3,3}$-minor-free graphs. We will present here a
more recent formulation of it, given by Thomas, Theorem~1.2
of~\cite{t99}. Roughly speaking the theorem states that a graph has no minor
isomorphic to $K_{3,3}$ if and only if it can be obtained from a planar graph or
$K_5$ by merging on an edge, a vertex, or taking disjoint components. To state
the theorem more precisely, we need the following definition of~\cite{t99}.

\begin{definition}
  Let $G_1$ and $G_2$ be graphs with disjoint vertex-sets, let $k \geq 0$ be an
  integer, and for $i=1,2$ let $X_i \subseteq V(G_i)$ be a set of pairwise
  adjacent vertices of size $k$. For $i=1,2$ let $G_i'$ be obtained by deleting
  some (possibly none) edges with both ends in $X_i$. Let $f : X_1 \to X_2$ be a
  bijection, and let $G$ be the graph obtained from the union of $G_1'$ and
  $G_2'$ by identifying $x$ with $f(x)$ for all $x \in X_1$. In those
  circumstances we say that $G$ is a \emph{$k$-sum} of $G_1$ and $G_2$.
\end{definition}

Now, we can state the theorem as a consequence of Wagner's theorem in the
following way.

\begin{theorem}[\cite{w37}, see also Theorem~1.2 of~\cite{t99}]
  \label{thm:Decomposition}
  A graph has no minor isomorphic to $K_{3,3}$ if and only if it can be obtained
  from planar graphs and $K_5$ by means of $0$-, $1$-, and $2$-sums.
\end{theorem}

Observe that for $2$-connected $K_{3,3}$-minor-free graphs we only have to take $2$-sums into
account as $0$- and $1$-sums do not yield a $2$-connected graph. In this way the decomposition of
Wagner fits perfectly well into a result of Walsh~\cite{w82} which delivers us -- similarly to the
case of planar graphs (see~\cite{bgw02}) -- with the necessary relations for our generating
functions.

The second ingredient for obtaining relations for our generating functions is to use a well-known
decomposition of a graph along its connectivity-structure, i.e.\ into connected, $2$-connected,
and $3$-connected components. Eventually, we obtain the following Lemma.

\begin{lemma}
  \label{lem:DecompositionConnectivityK33}
  Let $G(x, y)$, $C(x, y)$ and $B(x, y)$ denote the bivariate exponential
  generating functions for not necessarily connected, connected and
  $2$-connected $K_{3,3}$-minor-free graphs. Then we have
  \begin{eqnarray}
    G(x, y) = \exp\left(C(x, y)\right) \label{eqn:GeneralConnectivity}
    ~\text{ and }~
    C^{\bullet}(x, y) = x \exp\left(\deriv{x} B\left(C^{\bullet}(x,
    y), y \right) \right) \label{eqn:ConnectedConnectivity}.
  \end{eqnarray}
  Moreover, let $M(x, y)$ denote the bivariate generating function for
  $3$-connected planar maps which satisfies
  \begin{equation}
    \label{eqn:PlanarMaps}
    M(x, y) = x^2 y^2 \left( \frac{1}{1 + xy} + \frac{1}{1+y} - 1 -
    \frac{(1+U)^2(1+V)^2}{(1+U+V)^3} \right),
  \end{equation}
  where $U(x, y)$ and $V(x, y)$ are algebraic functions given by
  \begin{equation}
    \label{eqn:PlanarMapsUV}
    U = xy (1 + V)^2, \quad V = y (1 + U)^2,
  \end{equation}
  then we have
  \begin{equation}
    \label{eqn:2connectedDeriv}
    \deriv{y} B(x, y) = \frac{x^2}{2} \left( \frac{1 + D(x, y)}{1 + y} \right),
  \end{equation}
  where $D(x, y)$ is defined implicitly by $D(x, 0) = 0$ and
  \begin{equation}
    \label{eqn:3connected}
    \frac{M(x, D)}{2x^2D} + \frac{q x^3 D^9}{6} - \log\left(\frac{1+D}{1+y}\right) + \frac{x
    D^2}{1+x D} = 0,
  \end{equation}
  where $q$ marks the monomial for $K_5$.
\end{lemma}

\begin{proof}
  Equations~(\ref{eqn:GeneralConnectivity}) are standard and encode that a not
  necessarily connected graph consists of a set of connected graphs and a
  connected graph can be decomposed at a vertex into a set of $2$-connected
  graphs whose vertices can again be replaced by rooted connected graphs.  For
  more detailed proofs see for example~\cite{fs05}(p.95) and \cite{hp73}(p.10).

  Using Euler's polyhedral formula, Equations~(\ref{eqn:PlanarMaps}) and
  (\ref{eqn:PlanarMapsUV}) follow from~\cite{ms68}, where Mullin and
  Schellenberg derived the generating function for rooted $3$-connected planar maps
 according to the number of vertices and faces.

  Next, to derive the connection between $2$-connected and $3$-connected graphs,
  we can use a result of Walsh. More precisely, by Proposition~1.2 of~\cite{w82}
  we obtain Equations~(\ref{eqn:2connectedDeriv}) and~(\ref{eqn:3connected}),
  where we have to add only a monomial for $K_5$ compared to the class of planar
  graphs. For more details we refer to~\cite{bgw02}.
\end{proof}

%%% -------------------------------------------------
%%% --- Singular Expansions -------------------------
%%% -------------------------------------------------
\subsection{Singular Expansions and Asymptotic Estimates}
\label{sec:SingularExpansions}

We use the relations of the generating functions obtained so far to derive
singular expansions for the generating functions of the different connectivity
levels. We start from $3$-connected $K_{3,3}$-minor-free graphs and then go up
the connectivity hierarchy level by level. Eventually, this will allow us to
derive asymptotic estimates for the number of and properties of not necessarily
connected $K_{3,3}$-minor-free graphs in the subsequent sections.

We start with $3$-connected $K_{3,3}$-minor-free graphs. We have to
introduce only a slight modification into the formulas already known for planar
graphs (\cite{bgw02, gn05}).

From Lemma~\ref{lem:DecompositionConnectivityK33} we can derive analogously
to~\cite{bgw02} a singular expansion for $D(x,y)$. It will turn out that the
singularity of $D(x, y)$ changes only slightly compared to the case of
$2$-connected planar graphs, but yields a larger exponential growth rate.

To simplify expressions, we will use the following notation. The equation $Y(t)
= y$ has a unique solution in $t = t(y)$ in a suitable small neighbourhood of
$1$. Then we denote by $R(y) = \zeta(t(y))$. See Appendix~\ref{sec:AppendixA}
for expressions for $Y(t)$ and $\zeta$.

\begin{lemma}
  \label{lem:networks}
  For fixed $y$ in a small neighbourhood of $1$, $R(y)$ is the unique dominant
  singularity of $D(x, y)$. Moreover, $D(x, y)$ has a branch-point at $R(y)$,
  and the singular expansion at $R(y)$ is of the form
  \[ D(x, y) = D_0(y) + D_2(y) X^2 + D_3(y) X^3 + O(X^4)
  \]
  where $X = \sqrt{1 - x / R(y)}$ and the $D_i(y)$, $i=0,\ldots,3$ are given in
  Appendix~\ref{sec:AppendixA}.
\end{lemma}

To prove this lemma, one can mimic the proof of Lemma~3
in~\cite{bgw02}. Although we slightly changed the equations by adding a monomial
for $K_5$, one can check that the same arguments still hold.

Next, we solve Equation~(\ref{eqn:2connectedDeriv}) for $B(x, y)$ by integrating
according to $y$. We omit the proof as it follows closely the lines of proof of
Lemma~5 in~\cite{gn05}.

\begin{lemma}\label{Bgf}
  Let $W(x, z) = z (1 + U(x, z))$. The generating function $B(x, y)$ of
  $2$-connected $K_{3,3}$-minor-free graphs admits the following expression as a
  formal power series:
  \begin{equation}
    \label{eqn:2connected}
    B(x, y) = \beta(x, y, D(x, y), W(x, D(x, y))) + \frac{q x^5 D(x,y)^{10}}{120},
  \end{equation}
  where
  \[ \beta(x, y, z, w) = \frac{x^2}{2} \beta_1(x, y, z) - \frac{x}{4} \beta_2(x,
  z, w),
  \]
  and
  \begin{eqnarray*}
    \beta_1(x, y, z) & = & \frac{z(6x - 2 + xz)}{4x} + (1+z)
    \log\left(\frac{1+y}{1+z}\right) - \frac{\log(1+z)}{2} +
    \frac{\log(1+xz)}{2x^2} \\
    \beta_2(x, z, w) & = & \frac{2 (1 + x) (1 + w) (z+w^2) + 3(w-z)}{2(1+w)^2} -
    \frac{1}{2x} \log(1 + xz + xw + xw^2)\\
    & + & \frac{1-4x}{2x} \log(1+w) + \frac{1-4x+2x^2}{4x}
    \log\left(\frac{1-x+wz-xw+xw^2}{(1-x)(z+w^2+1+w)}\right).
  \end{eqnarray*}
\end{lemma}

We can use this lemma to obtain the singular expansion for $B(x, y)$.

\begin{lemma}
  \label{lem:2connectedExpansion}
  For fixed $y$ in a small neighbourhood of $1$, the dominant singularity of
  $B(x, y)$ is equal to $R(y)$. The singular expansion at $R(y)$ is of the form
  \begin{equation}
    \label{eqn:2connectedExpansion}
    B(x, y) = B_0(y) + B_2(y) X^2 + B_4(y) X^4 + B_5(y) X^5 + O(X^6)
  \end{equation}
  where $X = \sqrt{1 - x / R(y)}$, and the $B_i(y)$, $i=0, \ldots, 5$ are
  analytic functions in a neighbourhood of $1$.
\end{lemma}

\begin{proof}
  From Equation~(\ref{eqn:2connected}) we can see that for $y$ close to $1$ the
  only singularities come from the singularities of $D(x, y)$; hence the first
  claim of the theorem follows.

  The singular expansion for $B(x, y)$ can be obtained using
  Equation~(\ref{eqn:2connected}) and the singular expansion for $D(x, y)$. We
  substitute the singular expansion for $D(x, y)$, $U(x, D(x, y))$
  in~(\ref{eqn:2connected}). Then we set $x = \zeta(t) (1 - X^2)$ and $y = Y(t)$
  and expand the resulting expression. Now, collecting and simplifying the
  coefficients of the $X^i$ for $i=1,\ldots,5$ is a tedious calculation, but
  can be done with a computer algebra system such as {\sc Maple}. This yields
  the expressions for the $B_i(y)$ given in the appendix.
\end{proof}

For connected and not necessarily connected $K_{3,3}$-minor-free graphs, we can
derive singular expansions by carrying out an analogous calculation as in the
proof of Theorem~1 in~\cite{gn05}. We only have to adapt for the different
$D_i(y)$ and $B_i(y)$. One can easily check that the intermediate step of Claim~1
in~\cite{gn05} still holds and the rest of the calculations stays the same. The
coefficients of the expansions, which we obtain in this way, can be found in
Appendix~\ref{sec:AppendixA}.

\begin{lemma}
  \label{lem:connectedAndGeneralExpansion}
  For fixed $y$ in a small neighbourhood of $1$, the dominant singularity of
  $C(x, y)$ and $G(x, y)$ is equal to $\rho(y)$. The singular expansions at $\rho(y)$
  are of the form
  \begin{equation}
    \label{eqn:connectedExpansion}
    C(x, y) = C_0(y) + C_2(y) X^2 + C_4(y) X^4 + C_5(y) X^5 + O(X^6)
  \end{equation}
  and
  \begin{equation}
    \label{eqn:generalExpansion}
    G(x, y) = G_0(y) + G_2(y) X^2 + G_4(y) X^4 + G_5(y) X^5 + O(X^6)
  \end{equation}
  where $X = \sqrt{1 - x / \rho(y)}$, and the $C_i(y)$ and $G_i(y)$, $i = 0, \ldots,
  5$, are analytic functions in a neighbourhood of $1$.
\end{lemma}

From Lemmas~\ref{lem:2connectedExpansion} and~\ref{lem:connectedAndGeneralExpansion} we obtain the
following asymptotic estimates using the ``transfer theorem'', Corollary~VI.1 of~\cite{fs05}.

\begin{theorem}
  \label{thm:AsymptoticEstimates}
  Let $g_n$, $c_n$, and $b_n$ denote the number of not necessarily connected,
  connected and 2-connected resp.\ $K_{3,3}$-minor-free graphs on $n$
  vertices. Then it holds
  \begin{eqnarray}
    g_n & \sim & \alpha_g \; n^{-7/2} \; \rho_g^{-n} \; n!,
    \label{eqn:AsymptoticEstimateG}\\
    c_n & \sim & \alpha_c \; n^{-7/2} \; \rho_c^{-n} \; n!,
    \label{eqn:AsymptoticEstimateC}\\
    b_n & \sim & \alpha_b \; n^{-7/2} \; \rho_b^{-n} \; n!,
    \label{eqn:AsymptoticEstimateB}
  \end{eqnarray}
  where and $\alpha_g \doteq 0.42643 \cdot 10^{-5}$, $\alpha_c \doteq 0.41076
  \cdot 10^{-5}$, $\alpha_b \doteq 0.37074 \cdot 10^{-5}$, $\rho_c^{-1} =
  \rho_g^{-1} \doteq 27.22935$, and $\rho_b^{-1} \doteq 26.18659$ are
  analytically computable constants.
\end{theorem}

%%% -------------------------------------------------
%%% --- Structural Properties -----------------------
%%% -------------------------------------------------
\subsection{Structural Properties}
\label{sec:StructuralProperties}

If we consider a random $K_{3,3}$-minor-free graph, i.e.\ drawing a $K_{3,3}$-minor-free graph on
$n$ vertices uniformly at random from all such graphs on $n$ vertices, we can derive the following
properties using the algebraic singularity schema (Theorem~IX.10) of~\cite{fs05}.

\begin{theorem}
  \label{thm:ExpectedNumEdges}
  The number of edges in a not necessarily connected and connected random
  $K_{3,3}$-minor-free graph is asymptotically normally distributed with mean
  $\mu_n$ and variance $\sigma_n^2$, which satisfy
  \[ \mu_n \sim \kappa n\quad \textrm{ and } \quad \sigma_n^2 \sim \lambda n,
  \]
  where $\kappa \doteq 2.21338$ and $\lambda \doteq 0.43044$ are
  analytically computable constants.
\end{theorem}

Recall that we introduced the variable $q$ in the equations of the generating
functions above to mark the monomial for $K_5$. We can use this variable to
obtain a limit law for the number of $K_5$ used in the construction process
(following the above decomposition, see Theorem~\ref{thm:Decomposition}) of a
random $K_{3,3}$-minor-free graph. The next theorem shows that a linear number
of $K_5$ is used, but the constant is very small; this is exactly what one would
expect as the expected number of edges is only slightly larger than for planar
graphs (see Theorem~\ref{thm:ExpectedNumEdges} and~\cite{gn05}).

% \m{write more about what this ``number'' means?}

\begin{theorem}
  \label{thm:ExpectedNumK5s}
  Let $\mathrm{C}(G)$ denote the number of $K_5$ used in the construction of a
  random $K_{3,3}$-minor-free graph $G$ according to
  Theorem~\ref{thm:Decomposition}. Then $\mathrm{C}(G)$ is asymptotically
  normally distributed with mean $\mu_n$ and variance $\sigma_n^2$, which
  satisfy
  \[ \mu_n \sim \kappa n\quad \textrm{ and } \quad \sigma_n^2 \sim \lambda n,
  \]
  where $\kappa \doteq 0.92391 \cdot 10^{-4}$ and $\lambda \doteq 0.92440 \cdot
  10^{-4}$ are analytically computable constants. The same holds for a random
  connected $K_{3,3}$-minor-free graph.
\end{theorem}

\section{Graphs not containing $K_{3,3}^+$ as a minor}

In this brief section we give estimates for the number of graphs
not containing $K_{3,3}^+$ (the graph obtained from $K_{3,3}$ by
adding one edge) as a minor. For this we use the following recent
result from \cite{k33plus}.

\begin{theorem}
A 3-connected graphs not containing $K_{3,3}^+$ as a minor is
either planar or isomorphic to $K_{3,3}$ or $K_5$.
\end{theorem}

The analogous to Lemma \ref{Bgf} holds, that is we get $B(x,y)$
with an additional term (now we set $q=1$)
$$
 10 D(x,y)^9 {x^6 \over 6!}.
$$
This is because $K_{3,3}$ has 6 vertices, 9 edges, and 10
different labellings. The equation for  $D(x,y)$  in Lemma
\ref{lem:DecompositionConnectivityK33} has to be modified too in
order to take into account $K_{3,3}$, and one has to add a term
(again we set $q=1$) $x^4 D^8 /4$.

There is a corresponding expression for the singular coefficients
$B_i$ at the dominant singularity $R(y)$, which we do not write
down in detail in order to avoid repetition. We obtain the
dominant singularity $\rho(y)$ for the generating functions
$C(x,y)$ and $G(x,y)$, compute the corresponding singular
expansions and obtain the following result.

\begin{theorem}
%  \label{thm:AsymptoticEstimates}
  Let $g_n$, $c_n$, and $b_n$ denote the number of not necessarily connected,
  connected and 2-connected resp.\ $K_{3,3}^+$-minor-free graphs on $n$
  vertices. Then it holds
  \begin{eqnarray}
    g_n & \sim & \alpha_g \; n^{-7/2} \; \rho_g^{-n} \; n!,
%    \label{eqn:AsymptoticEstimateG}
\\
    c_n & \sim & \alpha_c \; n^{-7/2} \; \rho_c^{-n} \; n!,
%    \label{eqn:AsymptoticEstimateC}
\\
    b_n & \sim & \alpha_b \; n^{-7/2} \; \rho_b^{-n} \; n!,
%    \label{eqn:AsymptoticEstimateB}
  \end{eqnarray}
  where  $\rho_c^{-1} =
  \rho_g^{-1} \doteq 27.22948$, and $\rho_b^{-1} \doteq 26.18672$ are
  analytically computable constants.
\end{theorem}

Here is a table showing the approximate values of the growth
constants for planar, $K_{3,3}$-minor-free and
$K_{3,3}^+$-minor-free graphs.

\begin{table}[htb]
$$\begin{tabular}{|l|c|c|} \hline
Class of graphs & Growth constant & Growth constant for 2-connected \\
\hline
Planar &   27.22688 &  26.18486  \\
$K_{3,3}$-minor free   &  27.22935 & 26.18659 \\
$K_{3,3}^+$-minor free &  27.22948 & 26.18672 \\
\hline
\end{tabular}
$$
\end{table}

It is natural to ask if one can go further and treat the case
where the forbidden minor is obtained from $K_{3,3}$ by adding
\emph{two} edges. If the two edges share a vertex, then the
resulting graph is $K_{1,2,3}$; if the two edges do not share a
vertex, let us denote denote by $L$ the resulting graph.

Enumeration when we forbid $K_{1,2,3}$ is in principle feasible
along the previous lines because of the following theorem due to
Halin \cite{halin} (see also \cite[Section 6.1]{diestel}). The
Wagner graph $W$ consists of a cycle of length 8  in which
opposite vertices are adjacent.

\begin{theorem}
A 3-connected graph not containing $K_{1,2,3}$ as a minor is
either planar or isomorphic to  $K_5, W, L$ or to nine sporadic
non-planar graphs, or to a 3-connected subgraph of these.
\end{theorem}

This means that the generating function for 3-connected graphs not
containing $K_{1,2,3}$ as a minor  is obtained by adding a finite
number  of monomials to the generating function of 3-connected
planar graphs, which is already known to us, exactly as in the
previous section. However performing all the computations
corresponding to this case means: first to get the full list of
exceptional 3-connected graphs up to isomorphism from the previous
theorem, and in each case to compute the automorphism group in
order to determine the number of different labellings in each
case; and then to compute the singular expansions of all the
generating functions involved. We have refrained from doing these
computations, which would be along the same lines as before but
likely very cumbersome.

However, forbidding $L$ as a minor is another story and definitely
we cannot solve this problem at this stage. The reason is that in
this case one needs to take 3-sums (gluing along triangles) of
graphs in order to describe the family of 3-connected graphs not
containing $L$ as a minor \cite{diestel}, and we do not have the
necessary machinery to translate it into equations satisfied by
the generating functions. This problem already appears if we try
to count graphs not containing $K_5$ as a minor (notice that $L$
contains $K_5$ as a proper minor). We believe this is a
fascinating open problem that, if solved, will no doubt require
new ideas and techniques.

\medskip {\bf Acknowledgement.} The authors want to thank Eric Fusy
and Konstantinos Panagiotou for many helpful discussions and for
proofreading earlier versions of this manuscript.

%%% -------------------------------------------------
%%% --- include the bibliography file ---------------
%%% -------------------------------------------------
%\bibliographystyle{plain}
%\bibliography{refs}

%%% -------------------------------------------------
%%% --- Appendix A: Expressions ---------------------
%%% -------------------------------------------------
\appendix
\pagebreak

\section{Appendix}
\label{sec:AppendixA}

Here, we give the expressions for the coefficients of the singular expansions of
$D(x, y)$, $U(x, y)$, $B(x, y)$, $C(x, y)$ and $G(x, y)$ as well as the
expressions for the singularities. We use the same approach as in~\cite{bgw02}
and parametrize the expressions in the complex variable $t$.

The variable $q$ used for counting the number of $K_5$ appears
explicitly only in the expression for $h$; the reason is that this
is the only place where it is needed for computing the radius of
convergence, which in turn is needed for estimating the expected
value and variance in Theorem \ref{thm:ExpectedNumK5s}.

\begin{eqnarray*}
  h & = & \frac{t^2}{8192 {(3t+1)^6 (2t+1) (t+3)}} \left(13122 q t^9 + 45927 q t^8 -
  1658880 t^7 + 19683 q t^7 \right. \\
  & & - 12496896 t^6 - 8847360 t^5 + 6832128 t^4 +
  \left. 10399744 t^3 + 4739072 t^2 + 958464 t + 73728 \right) \\
  Y(t) & = & - \frac{2t + 1}{(3t + 1) (t-1)} e^{-h} - 1 \\
  \zeta & = & - \frac{(t - 1)^3 (3t + 1)}{16 t^3} \\
  Q & = & 78732 t^9 - 1328940 t^8 - 26889705 t^7 - 153744066 t^6 - 415828997 t^5
  - 522964992 t^4 \\
    & & - 342073344 t^3 - 121237504 t^2 - 22151168 t - 1638400 \\
  K & = & 78732 t^{11} + 472392 t^{10} - 2668221 t^9 - 816345 t^8 + 92026557 t^7
  + 562023429 t^6 \\
  & & + 1040556032 t^5 + 926367744 t^4 + 455663616 t^3 + 127336448 t^2 +
  19005440 t + 1179648 \\
  U_0 & = & \frac{1}{3t} \\
  U_1 & = &  - \left(-\frac{2}{27} \frac{(3t + 1) K}{t^3 (t + 1) Q}
  \right)^{\frac12} \\
  U_2 & = & - \frac{(3t + 1)^2}{54 t^2 (t + 1)^2 Q^2} \left(
  6198727824 t^{20} + 180231719760 t^{19} + 891036025560 t^{18} \right. \\
  & & \left. - 12902936763600
  t^{17} - 197722264231071 t^{16} - 1821396525148269 t^{15} \right. \\
  & & - 13816272361145022 t^{14} - 79424397121737354 t^{13} - 324711461744767867 t^{12} \\
  & &  -931873748086896665 t^{11} - 1881275802907541504 t^{10} -
  2713502925437276160 t^9 \\
  & & - 2843653010633469952 t^8 - 2190731661037666304 t^7 - 1246514524950953984
  t^6 \\
  & & - 521994799964094464 t^5 - 158674913803108352 t^4 - 34025665074298880
  t^3\\
  & & \left. - 4876321721155584 t^2 - 418948289921024 t - 16312285790208 \right) \\
  D_0 & = & - \frac{3t^2}{(3t + 1) (t - 1)} \\
  D_1 & = & 0 \\
  D_2 & = & - \frac{t (2t + 1)^2}{(3t + 1) (t - 1) Q} \left(19683 t^8 + 118098
  t^7 - 1592325 t^6 - 10616832 t^5 - 30670848 t^4 \right. \\
  & & \left. + 7602176 t^3 + 24444928 t^2 + 9830400 t + 1179648 \right) \\
  D_3 & = & \frac{131072}{9 Q^2} \left(  \left( -\frac{(3t +1) K}{t^3 (t+1) Q}
  \right)^{\frac12} \sqrt{6} t^2 (3t + 1) (t + 3)^2 (2t + 1)^2 K \right)
\end{eqnarray*}

\begin{eqnarray*}
%%   P_1 & = & 1549681956 t^{24}  - 68328432252 t^{23} - 646991330895 t^{22}
%%   + 1383569088336 t^{21} \\
%%   & & - 57934645367238 t^{20} - 1030641858893628 t^{19} - 5581315778170878
%%   t^{18} \\
%%   & & - 9690527546116164 t^{17} + 14823917538797880 t^{16} + 66591676440148968
%%   t^{15} \\
%%   & & - 6807229356797163 t^{14} - 121180156627243452 t^{13}- 38691868953118942
%%   t^{12} \\
%%   & & + 93938978979606528 t^{11} + 65141137737269248 t^{10} - 21686663626104832 t^9 \\
%%   & & - 36470289308778496 t^8 - 9659501232001024 t^7 + 4668686142685184 t^6  \\
%%   & &   + 4119895696351232 t^5  + 1329802690691072 t^4 + 223343466774528 t^3 \\
%%   & & + 17853474406400 t^2 + 207232172032 t - 40265318400\\
%
  P_1 & = &  1549681956{t}^{19}-60580022472{t}^{18}-965388262815{t}^{17}-
2822075181459{t}^{16} \\
      &   &  -63004687280883{t}^{15}-1326793976317287{t}
^{14}-11608693177471470{t}^{13} \\
      &   & -55082955555464994{t}^{12} -157459666865762304{t}^{11}-279393068914421760{t}^{10} \\
      &   & -323288788914892800{t}^{9}-254483996115259392{t}^{8} -139939270751358976{t}^{7} \\
      &   & -54299625067175936{t}^{6}-14753365577572352{t}^{5}-2718756694392832{t}^{4} \\
      &   & -314310035243008{t}^{3}-18285655490560{t}^{2}-5905580032t+40265318400 \\
  P_2 & = & -472392 t^{12}-2991816 t^{11}+15064542 t^{10}+10234512
  t^9-550526652 t^8\\
  & & -3556193688 t^7-7367383050 t^6-7639318528 t^5-4586717184
  t^4-1675345920 t^3\\
  & & -368705536 t^2-45088768 t -2359296 \\[4ex]
     B_0 & = & \frac{1}{4}\ln(3+t)
               -\frac{(3t+1)^2(-1+t)^6\ln(2t+1)}{1024 t^6}
               -\frac{(3t^4-16t^3+6t^2-1)\ln(3t+1)}{32 t^3} \\
         & &   -\frac{1}{2}\ln(t)-\frac{3}{2}\ln(2)
               +\frac{(3t-1)^2(1+t)^6\ln(1+t)}{512 t^6}
  -{\frac {\left( t-1 \right) ^{2}}{41943040 {{t}^{4}
      \left( 3t+1 \right) ^{5} \left( t+3 \right) }}} \\
%%   B_0 & = & \frac{1}{4{t}^{6}}\, \left( -{\frac {9}{256}}\, \left( t+\frac13 \right)^{2} \left( t-1
%%   \right) ^{6}\ln  \left( {\frac {-2\,t-1}{3\,{t}^{2}-2\,t-1}} \right) \right. \\
%%   & & + \left( -{\frac {3}{32}}\,{t}^{7}-{\frac {9}{512}}\,{t}^{8}+{\frac {7
%%     }{128}}\,{t}^{6}+\frac{1}{32}\,{t}^{3}-{\frac {15}{256}}\,{t}^{4}-\frac{3}{16}\,{t}^{5
%%   }-{\frac {1}{512}}+{\frac {3}{128}}\,{t}^{2} \right) \ln  \left( {
%%     \frac { \left( 3\,t+1 \right) ^{2} \left( t-1 \right) ^{2}}{ \left( t+
%%       1 \right) ^{4}}} \right) \\
%%   & & + \left( {t}^{3}\ln  \left( 1+\frac{3}{16}\,{\frac {
%%       \left( t-1 \right) ^{2}}{t}} \right) +\frac12\,{t}^{3}\ln  \left( \frac{1}{16}\,{
%%     \frac { \left( t+1 \right) ^{2} \left( 3\,t+1 \right) }{{t}^{2}}}
%%   \right) \right. \\
%%   & & \left. \left. -\frac38\, \left( {t}^{4}-\frac43\,{t}^{3}+2\,{t}^{2}-\frac13 \right)
%%   \ln  \left( - \left( t-1 \right) ^{-1} \right)  \right) {t}^{3}
%%   \right) \\
  & &
   \left( 19683{t}^{13} -131220{t}^{12} -183708{t}^{11} +
  360921744{t}^{10} + 2005423731{t}^{9} + 3887177580{t}^{8}+\right . \\
   & & \left. 5603033310{t}^{7}+ 4821770240{t}^{6}+ 2013921280{t}^{5}
      +229048320{t}^{4} -97157120{t}^{3} \right. \\
   & & \left. -31436800{t}^{2} -2048000t+122880\right) \\
  B_1 & = & 0 \\
  B_2 & = &
           -\frac{(3t-1)(3t+1)(1+t)^3(-1+t)^3\ln(1+t)}{256t^6}
           +\frac{(3t+1)^2(-1+t)^6\ln(2t+1)}{512t^6} \\
   & &  +\frac{(3t+1)(-1+t)^3\ln(3t+1)}{32t^3}
%%   {\frac {9 \left( t+\frac13 \right)  \left( t-1
%%   \right) ^{3}}{1024  {t}^{6}}} \left( 2 \left( t+\frac13 \right)  \left( t-1
%%   \right) ^{3}\ln  \left( {\frac {-2t-1}{3{t}^{2}-2t-1}} \right)
%%   \right. \\
%%   & & \left.  + \left( t-\frac13 \right)  \left( t+1 \right) ^{3}\ln  \left( {\frac {
%%       \left( 3t+1 \right) ^{2} \left( t-1 \right) ^{2}}{ \left( t+1
%%       \right) ^{4}}} \right) +{\frac {32}{3}}{t}^{3}\ln  \left( - \left(
%%  t-1 \right) ^{-1} \right)  \right) \\
   +{\frac {\left( t-1
      \right) ^{4}}{8388608 {{t}^{4} \left( t+3 \right)  \left( 3t+1 \right) ^{5}
  }}}  \left( 19683{t}^{11}-13122{t}^{10} \right. \\
 & & -190269{t}^{9}+122862096{t}^{8} +626914188{t}^{7}+555393024{t}^{6}+28803072{t}^{5} \\
 & & \left. -163438592{t}^{4}-81084416{t}^{3} -14852096{t}^{2}-720896t+ 49152 \right) \\
  B_3 & = & 0 \\
  B_4 & = & -\frac{(-1+t)^5P_1}{8388608 t^4 (t+3) (3t + 1)^5 Q}  - \frac{9 (t +
  \frac13)^2 (t - 1)^6 (-2 \ln(t+1) + \ln(2t + 1))}{1024 t^6} \\
  B_5 & = & - \frac{\sqrt{\frac{3 P_2}{t^3 (t + 1) Q}} P_2^2 (t-1)^6}{2880
  (3t + 1)^5 (t+1) t Q^2}
\end{eqnarray*}

\begin{align*}
  C_0 & =  R + B_0 + B_2     & G_0 & = \exp(C_0)\\
  C_1 & =  0                 & G_1 & = 0\\
  C_2 & =  -R                & G_2 & = \exp(C_0)C_2\\
  C_3 & =  0                 & G_3 & = 0\\
  C_4 & =  -\frac{1}{2}\left(R + \frac{R^2}{2 B_4 - R}\right)
                             & G_4 & = \exp(C_0) \left(C_4 + \frac{C_2^2}{2} \right) \\
  C_5 & =  B_5 \left(1 - \frac{2 B_4}{R}\right)^{-\frac52}
                             & G_5 & = \exp(C_0) C_5
\end{align*}

\end{document}